\documentclass [intlimits, 12pt]{article}
\usepackage[mathscr]{eucal}
\usepackage{amsmath}
\usepackage{amsthm}
\usepackage{amssymb}
\usepackage{amsfonts} 
\usepackage{mathrsfs}
\usepackage{amscd}
\usepackage{enumerate} 
\usepackage{color}
\usepackage{layout}
\usepackage{stmaryrd} 
\usepackage{float} 
\usepackage{longtable}
\usepackage{mathtools} 
\usepackage[pdftex]{graphicx}

\usepackage{setspace} 

\newtheorem{Theorem}[equation]{Theorem}

\newtheorem{Proposition}[equation]{Proposition}
\newtheorem{Corollary}[equation]{Corollary}
\newtheorem{Lemma}[equation]{Lemma}

\newtheorem{Remark}[equation]{Remark}

\newtheorem{Definition and Proposition}[equation]{Definition and Proposition}
\newtheorem{Definition and Theorem}[equation]{Definition and Theorem}
\newtheorem{Definition and Remark}[equation]{Definition and Remark}
\newtheorem{Definition and Lemma}[equation]{Definition and Lemma}
\newtheorem{Definition and Example}[equation]{Definition and Example}

\newtheorem*{IntroTheoremA}{Theorem A}
\newtheorem*{IntroTheoremB}{Theorem B}
\newtheorem*{IntroTheoremC}{Theorem C}

\newcommand{\purge}[1]{} 

\newcommand{\vsp}{\vspace{2mm}}

\def\epsilon{\varepsilon}
\def\phi{\varphi}

\newcommand{\ga}{\gamma}

\newcommand{\lam}{\lambda}

\newcommand{\om}{\omega}

\newcommand{\gadot}{{\dot{\ga}}}

\newcommand{\pti}{{\ti{p}}}

\newcommand{\uti}{{\ti{u}}}

\newcommand{\zdot}{{\dot{z}}}

\newcommand{\Mhat}{{\hat{M}}}

\def\N{{\mathbb N}}

\def\R{{\mathbb R}}

\def\Q{{\mathbb Q}}
\def\Z{{\mathbb Z}}

\newcommand{\mcH}{\mathcal H}

\newcommand{\mcM}{\mathcal M}

\newcommand{\mcMhat}{\widehat{\mathcal M}}

\newcommand{\mfc}{\mathfrak c}

\newcommand{\mfh}{\mathfrak h}

\newcommand{\ti}{\tilde}
\newcommand{\x}{\times}
\newcommand{\del}{\partial}

\newcommand{\beq}{\begin{equation}}
\newcommand{\eeq}{\end{equation}}
\newcommand{\beqs}{\begin{equation*}}
\newcommand{\eeqs}{\end{equation*}}

\DeclarePairedDelimiter\abs{\lvert}{\rvert}

\newcommand{\mcHpr}{{\mathcal H}_{pr}}
\newcommand{\mcHprti}{{\ti{\mathcal H}_{pr}}}

\DeclareMathOperator{\Crit}{Crit}

\DeclareMathOperator{\Diff}{Diff}

\DeclareMathOperator{\Fix}{Fix}

\DeclareMathOperator{\grad}{grad}

\DeclareMathOperator{\Ham}{Ham}

\DeclareMathOperator{\Id}{Id}
\DeclareMathOperator{\Img}{Im}
\DeclareMathOperator{\Ind}{Ind}

\DeclareMathOperator{\rk}{rk}

\DeclareMathOperator{\Symp}{Symp}

\def\slashii#1{\setbox0=\hbox{$#1$}             
\dimen0=\wd0                                 
\setbox1=\hbox{\sl/} \dimen1=\wd1            
\ifdim\dimen0>\dimen1                        
\rlap{\hbox to \dimen0{\hfil\sl/\hfil}}   
#1                                        
\else                                        
\rlap{\hbox to \dimen1{\hfil$#1$\hfil}}   
\hbox{\sl/}                               
\fi}                                         %
\def\slashiii#1{\setbox0=\hbox{$#1$}#1\hskip-\wd0\hbox to\wd0{\hss\sl/\/\hss}}
\newcommand{\refpositionprim}{Lemma \ref{position prim}}

\newcommand{\reffirstIntersection}{Lemma \ref{firstIntersection}}
\newcommand{\refrefineclassif}{Proposition \ref{refineclassif}} 
\newcommand{\refoneiterate}{Theorem \ref{oneiterate}}
\newcommand{\refplace}{Corollary \ref{place}}

\newcommand{\refhomoclinicMorse}{Theorem \ref{homoclinicMorse}}

\newcommand{\refnotorsion}{Theorem \ref{notorsion}}

\newcommand{\refquotient}{Lemma \ref{quotient}}

\begin{document}

\title{Computational complexity, torsion-freeness of homoclinic Floer homology, and homoclinic Morse inequalities.}

\author{Sonja Hohloch \\ {\small sonja.hohloch@uantwerpen.be}\\ {\small Department of Mathematics and Computer Science} \\ {\small University of Antwerp} \\ {\small Middelheimlaan 1} \\ {\small B-2020 Belgium}}

\date{\today}

\maketitle

\begin{abstract}
\noindent
Floer theory was originally devised to estimate the number of 1-periodic orbits of Hamiltonian systems.
In earlier works, we constructed Floer homology for homoclinic orbits on two dimensional manifolds using combinatorial techniques. In the present paper, we study theoretic aspects of computational complexity of homoclinic Floer homology. More precisely, for finding the homoclinic points and immersions that generate the homology and its boundary operator, we establish sharp upper bounds in terms of iterations of the underlying symplectomorphism. This prepares the ground for future numerical works.

\noindent
Although originally aimed at numerics, the above bounds provide also purely algebraic applications, namely
\begin{enumerate}[1)]
\setstretch {0.5}
 \item 
 torsion-freeness of primary homoclinic Floer homology,
 \item
 Morse type inequalities for primary homoclinic orbits.
\end{enumerate}
\end{abstract}


\section{Introduction}

This section is subdivided into four parts: First we recall some essential notions from homoclinic dynamics, then we give a brief introduction to classical Floer theory. Subsequently we explain the intuition behind homoclinic Floer homology before we summarize the main results of the present paper. 


\subsection{Notions in homoclinic dynamics}

The orbit type we are interested in are the so-called homoclinic orbits whose definition we will recall in the following: Let $N$ be a manifold and $f\in \Diff(N)$ a diffeomorphism. $x \in N$ is an {\em $m$-periodic point} if there is $m \in \N$ such that $f^m(x)=x$. For $m=1$, such an $x$ is usually called a {\em fixed point} and the set of fixed points is denoted by $\Fix(f)$.

A fixed point $x$ is called {\em hyperbolic} if the eigenvalues of the linearization $Df(x)$ of $f$ in $x$ have modulus different from $1$. The {\em stable manifold} of a hyperbolic fixed point $x$ is given by $W^s(f,x):= \{p \in N \mid\lim_{n \to \infty} f^n(p)=x\}$ and the {\em unstable manifold} is given by $W^u(f,x):= \{p \in N \mid\lim_{n \to -\infty} f^n(p)=x\}$. 
The connected components of $W^s(f,x) \backslash \{x\}$ resp.\ $W^u (f,x) \backslash \{x\}$ are called the {\em branches} of $W^s(f,x)$ resp.\ $W^u(f,x)$. 

A diffeomorphism $f$ is called {\em $W$-orientation preserving w.r.t.\ $x \in \Fix(f)$} if each branch of the stable and unstable manifolds of $x$ is mapped to itself.

The intersection points of the stable and unstable manifold of $x$ are called {\em homoclinic points of $x$} and we denote the set of homoclinic points of $x$ by $\mcH(f,x) := W^s(f,x) \cap  W^u(f,x)$. An example is sketched in Figure \ref{tangle}.

\begin{figure}[h]
\begin{center}

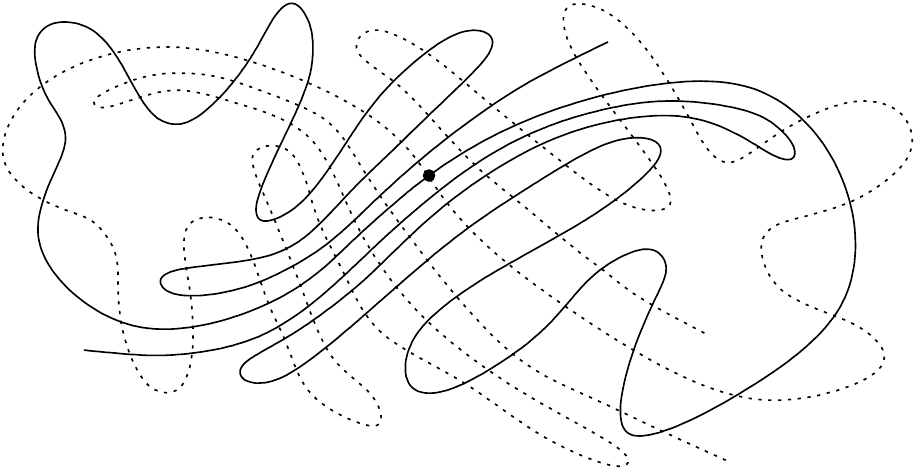

\caption{The intersection behaviour of transversely intersecting stable and unstable (dotted) manifold of a hyperbolic fixed point.}
\label{tangle}

\end{center}
\end{figure}

An {\em orbit} associated to a point $p \in N$ is the set $\{ f^n(p) \mid n \in \Z\}$. If $p$ is a periodic resp.\ homoclinic point, we call the orbit periodic resp.\ homoclinic. Obviously, the stable and unstable manifolds are invariant under the action 
$$\Z \x N \to N, \quad (m, p) \mapsto f^m(p).$$

Homoclinic points are somehow the `next more complicated' orbit type after fixed points and periodic points. The existence of (transverse) homoclinic points was discovered by Poincar\'e \cite{poincare1}, \cite{poincare2} around 1890 when he worked on the $n$-body problem. In 1935, Birkhoff \cite{birkhoff} noticed the existence of high-periodic points near homoclinic ones, but it took until Smale's horseshoe formalism in the 1960s to obtain a formal and precise description of the implied dynamics. Since then, homoclinic points have been studied by various means like perturbation theory, calculus of variations and numerical approximation, but many questions are still open.


\subsection{Classical Floer theory}

In order to introduce Floer theory we need some preparations.
A smooth 2n-dimensional manifold $M$ is {\em symplectic} if it admits a closed nondegenerate 2-form $\om$. For example, surfaces equipped with their volume forms are symplectic.  
The class of transformations associated with symplectic geometry are those diffeomorphisms that leave the symplectic form $\om$ invariant, namely the group of {\em symplectomorphisms} $\Symp(M,\om):= \{ f \in \Diff(M) \mid f^*\om = \om\}$. There is a subgroup that is particularly important for Floer theory, namely the group of Hamiltonian diffeomorphisms $\Ham(M, \om)$ which is defined as follows.
Given a smooth function $F: M \x \mathbb S^1 \to \R$, we set $F_t:=F(\cdot, t)$ and define its (nonautonomous) {\em Hamiltonian vector field} $X_t^F$ via $\om(X_t^F, \cdot)=-dF_t(\cdot)$. Then $\zdot(t)=X^F_t(z(t))$ is the associated {\em Hamiltonian equation} and its (nonautonomous) flow is called {\em Hamiltonian flow}. A {\em Hamiltonian diffeomorphism} is a symplectomorphisms which can be written as the time-1 map $\phi_1$ of a Hamiltonian flow $\phi_t$. A Hamiltonian diffeomorphism is called {\em nondegenerate} if its graph intersects the diagonal in $M \x M$ transversely.

\vsp

Since symplectic geometry provides the framework for Hamiltonian systems it shows up naturally in physics, but, since the 1960s, symplectic geometry has also been studied for its own sake. Moser and others investigated for instance the distinction between symplectic and volume preserving geometry: in dimension two, being symplectic is the same as being volume preserving, but in dimension strictly higher than two, it differs. Symplectomorphisms are volume preserving w.r.t. the volume form $\om^n:=\om \wedge \dots \wedge \om$, but not all volume preserving maps preserve a symplectic form.

\vsp

V.\ I.\ Arnold conjectured in the 1960s that the number of fixed points of a nondegenerate Hamiltonian diffeomorphism on a closed, symplectic manifold is greater or equal to the sum over the Betti numbers of the underlying manifold. 
Arnold's conjecture was open for a long time until it was proven for the $2n$-dimensional torus by Conley and Zehnder in 1983. Floer \cite{floer1, floer2, floer3} achieved a breakthrough by turning the fixed point problem into an {\em intersection problem}:
he considered the fixed points of a Hamiltonian diffeomorphism as intersection points of the graph of the Hamiltonian diffeomorphism with the diagonal in the symplectic manifold $(M \x M, \om \oplus (-\om))$.
In this setting, the diagonal and the graph turn out to belong to a special class of submanifolds, the so-called {\textit{Lagrangian submanifolds} (submanifolds on which the symplectic form vanishes and whose dimensions are half the dimension of $M$); and most important, Lagrangian submanifolds have good properties concerning (Fredholm) analysis.
Moreover, the intersection points of the graph and the diagonal can be seen as critical points of the symplectic action functional. Floer considered this functional as some kind of Morse function and went along to devise some kind of `infinite dimensional Morse theory' for the symplectic action functional. This theory (and its generalizations) is nowadays known as {\em Floer theory}. The associated homology theory is referred to as {\em Floer homology}.
Apart from leading to a proof of Arnold's conjecture, Floer theory gave rise to many other applications in symplectic geometry, dynamical systems and other fields of mathematics and is vividly studied nowadays.

\vsp

Roughly, the construction of Lagrangian Floer homology goes as follows. For details we refer the reader to Floer's original works \cite{floer1}, \cite{floer2}, \cite{floer3} and Fukaya $\&$ Oh $\&$ Ohta $\&$ Ono \cite{fo3a}, \cite{fo3b}. 

Consider two transversely intersecting, `sufficiently nice' Lagrangian submanifolds $L$ and $L'$ lying in a `sufficiently nice' symplectic manifold $ (M, \om)$. Then, to two intersection points $p,q \in L \cap L'$, we can assign a relative index $I(p,q) \in \Z$. By fixing one $\pti \in L \cap L'$ as reference point one obtains an index $I(p):=I(p,\pti) \in \Z$. Now define the $k$th Floer chain group (with $\Z\slash 2 \Z$-coefficients) as the free group generated by all intersection points of index $k$, i.e.\
$$
CF_k:= CF_k(L, L'):= \bigoplus_{\stackrel{p \in L \cap L'}{I(p)=k}} \Z\slash 2 \Z \ p
$$
A complex structure that varies with its footpoint is usually called an {\em almost complex structure}. The corresponding generalisation of holomorphic maps are called {\em pseudo-holomorphic} maps. Now consider $p^-, p^+ \in L \cap L'$. The space of pseudo-holomorphic maps $u: \R \x [0,1] \to M$ with $u(\R \x \{0\}) \subseteq L$ and $u(\R \x \{1\}) \subseteq L'$ satisfying $\lim_{s \to -\infty}u(s, \cdot)=p^-$ and $\lim_{s \to +\infty}u(s, \cdot)=p^+$ is denoted by $M(p^-, p^+)$. This space carries the $\R$-action
$$
\R \x M(p^-, p^+) \to M(p^-, p^+), \qquad (r, u) \mapsto u( \cdot + r, \cdot).
$$
Dividing by this action yields 
$$
\Mhat(p^-, p^+):= M(p^-, p^+) \slash \R
$$
which has in fact dimension $I(p^-, p^+)-1$. For $I(p^-, p^+)=1$, it is zero dimensional. Being compact, it has thus cardinality $|\Mhat(p^-, p^+)| <\infty$. Counting modulo 2 to be compatible with the $\Z \slash 2 \Z$-coefficients, we define the boundary operator
$$
d_* : CF_* \to CF_{*-1}, \qquad d p^- := \sum_{\stackrel{ p^+ \in L \cap L'}{I(p^-, p^+)=1}} (|\Mhat(p^-, p^+)|\ mod\ 2) \ p^+
$$
on the generators and extend it by linearity. It satisfies $d_{*-1} \circ d_* =0$ turning $(CF_*, \del_*)$ into a chain complex. The associated homology
$$
HF_*(L, L'):= \ker d_* \slash \Img d_{*+1} 
$$
is called {\em Lagrangian Floer homology} of $L$ and $L'$.


\subsection{The motivation for homoclinic Floer theory}

How does homoclinic Floer homology link homoclinic points to Floer theory? Remember that homoclinic points and Floer theory both involve {\em intersecting} submanifolds. Thus we have to check if the intersection problem of stable and unstable manifolds fits the requirements of Floer's setting, i.e. are they Lagrangian? The answer is yes, under certain conditions: If we work with symplectomorphisms instead of just diffeomorphisms, the stable and unstable manifolds are always Lagrangian. But in Floer's setting (and its generalizations), the Lagrangians are usually compact or at least `sufficiently nice'. Unfortunately the stable and unstable manifold are usually only injectively immersed and give rise to an abundance of intersection points as sketched in Figure \ref{tangle}.

Classical Floer theory knows certain techniques to deal with `nice non-compactness', but they fail for stable and unstable manifolds --- there are just `too many' intersection points. Nevertheless, we will see in Section \ref{types} how one can define a Floer theory for homoclinic points.

\vsp

Up to our knowledge, there are few works apart from our papers \cite{hohloch1, hohloch2} where homoclinic orbits are studied with symplectic methods or means related to Floer theory: Hofer $ \& $ Wysocki \cite{hofer-wysocki} use pseudo-holomorphic curves and Fredholm theory. Cieliebak $\&$ S\'er\'e \cite{cieliebak-sere} combine variational techniques and pseudo-holomorphic curves. 
And Lisi \cite{lisi} generalizes Coti Zelati $\&$ Ekeland $\&$ S\'er\'e \cite{coti zelati-ekeland-sere} using Lagrangian embedding techniques. 


\subsection{Main results}

Below in Section \ref{types}, we will describe in detail the construction of homoclinic Floer homology. Briefly, in order to compute homoclinic Floer homology, we have to locate the generators of the homoclinic chain groups and then compute the boundary operator. The generators are the so-called primary homoclinic points which are defined in \eqref{primary}. The computation of the boundary operator involves finding and counting certain immersions joining two primary homoclinic points. To be more precise, consider a $W$-orientation preserving $\phi \in \Symp(\R^2)$ with $x \in \Fix(\phi)$ hyperbolic.
The generators, i.e.\ the primary homoclinic points, are intersection points of the associated (un)stable manifolds. Since the (un)stable manifolds are invariant under the $\Z$-action induced by $\phi$, the orbit of a generator `runs along the whole length' of the (un)stable manifolds. But these are noncompact and only injectively immersed, not embedded. Therefore, a priori, we do not have any control over the whereabouts of the generators and thus in particular not over immersions joining two of them. Without upper bounds on the `distance' between generators joint by immersions, finding these generators and immersions is hopeless.

The central theorem of this paper improves and maximally strengthens a result from an earlier work (to be precise, Proposition 26 and Lemma 27 in \cite{hohloch1}): 
In Hohloch \cite{hohloch1}, we proved the {\em existence} of an upper bound, but could not give any estimates for it. The present paper gives a sharp upper bound in terms of iterations of $\phi$ and, in addition, also pins down certain signs associateed to the immersions. That way, homoclinic Floer homology becomes accessible and managable for numerics. 

\begin{IntroTheoremA}
Let $\phi \in \Symp(\R^2)$ be $W$-orientation preserving with $x \in \Fix(\phi)$ hyperbolic and use the convention $\phi^0=\Id$. 
 Let $p$ be the startpoint and $q$ be the endpoint of an immersion used in the boundary operator of primary homoclinic Floer homology. Then:
 \begin{enumerate}[1)]
   \item 
  The endpoint $q$ lies between $\phi^{-1}(p)$ and $\phi(p)$ on the stable and/or unstable manifold. 
    \item
  If there exists in addition an immersion with start point $p$ and endpoint $\phi^n(q)$ for some $n \in \Z^{\neq 0}$ then either $n=1$ or $n=-1$. In both cases, the immersions carry opposite signs w.r.t.\ the original one.
   \end{enumerate}
\end{IntroTheoremA}

\noindent
The above theorem paraphrases the content of \refoneiterate\ and \refplace. The bound in Theorem A is surprisingly low: the endpoint of an immersion is maximally just plus/minus one `iteration interval' away from its start point. Together with the fact that also each primary homoclinic orbit hits an `iteration interval' exactly once (cf.\ \refpositionprim) this considerably reduces the parts of the (un)stable manifolds that have to be searched. The proof of Theorem A is not difficult, but quite tedious since one has to check a large number of (sub)cases.

Although Theorem A was originally aimed at numerical computations of homoclinic Floer homology, it has nevertheless interesting algebraic applications.

\begin{IntroTheoremB}
 Primary homoclinic Floer homology of $W$-orientation preserving symplectomorphisms over $\Z$ is torsion-free.
\end{IntroTheoremB}

This will be proven in \refnotorsion. The universal coefficient theorem in homological algebra describes by means of torsion the dependence of homology groups on the chosen coefficient ring. Thus torsion-freeness implies that, for instance, homoclinic Floer homology computed with $\Z$-coefficients is the same as with $\Q$- or $\R$-coefficients. 

\vsp

Since every finitely generated abelian group has a direct sum decomposition in a finitely generated free subgroup and a unique torsion group, therefore torsion-freeness simplifies inequalities and estimates involving the rank of homology groups. The most prominent examples for such inequalities are the Morse inequalities induced by Morse homology (recalled in Section \ref{classicalMorse}). Homoclinic Floer homology also gives rise to such inequalities:

\begin{IntroTheoremC}
 There are Morse type inequalities for primary homoclinic points.
\end{IntroTheoremC}

This summarizes \refhomoclinicMorse. Geometrically these inequalities relate and estimate the number of primary homoclinic points with different Maslov indices.


\subsection*{Acknowledgements}

The author wishes to thank Wim Vanroose for insights into numerics.


\section{Primary homoclinic Floer homology}

\label{types}


There exist four types of Floer theory generated by homoclinic points as shown in Hohloch \cite{hohloch1, hohloch2}. Each of them has different flavours and properties, but for the present paper only the following one is relevant.

\vsp

Let $(M, \om)$ be a symplectic manifold and $\phi$ a symplectomorphisms with hyperbolic fixed point $x$ and transversely intersecting (un)stable manifolds $W^s:=W^s(\phi, x)$ and $W^u:=W^u(\phi, x)$.
As already mentioned above, $W^s$ and $W^u$ are `highly noncompact' (cf.\ Figure \ref{tangle}) which poses a problem for the analysis part of classical Floer theory. Fortunately there is a way to avoid this obstacle. If we restrict our studies to {\em two-dimensional} manifolds we may replace the involved analysis by combinatorics as proved by several authors (cf.\ de Silva \cite{de silva}, Felshtyn \cite{felshtyn}, Gautschi $\&$ Robbin $\&$ Salamon \cite{gautschi-robbin-salamon}). Restricting to dimension two simplifies only the trouble with the analysis -- the difficulties related to the abundance of intersection points remain unchanged.

\vsp	

Now assume $(M, \om)$ to be $\R^2$ or a closed surface of genus $g \geq 1$ with their resp.\ volume forms. This implies in particular that the (un)stable manifolds are {\em one dimensional}. 
Consider the set of homoclinic points $\mcH :=W^s \cap W^u$ where we assume the intersection to be transverse. Let $p$, $q \in \mcH$ and denote by $[p,q]_i $ the (one dimensional!) segment between $p$ and $q$ in $W^i$ for $i \in \{s,u\}$. 
The symplectomorphism $\phi$ introduces a $\Z$-action $\mcH \x \Z \to \mcH$, $(p,n) \mapsto \phi^n(p)$ on the set of homoclinic points. For transversely intersecting $W^s \cap W^u$, the sets $\mcH$ and $\mcH \slash \Z$ are {\em both infinite} as a glance at Figure \ref{tangle} shows. 
Denote by $c_p: [0,1] \to W^u \cup W^s$ a continuous curve with $c_p(0)=x=c_p(1)$ which runs through $[x,p]_u$ to $p$ and through $[p,x]_s$ back to $x$.
We define the homotopy class of $p$ via $[p]:= [c_p] \in \pi_1(M, x)$. Then $\mcH_{[x]}:=\{p \in \mcH \mid [p] =[x] \}$ is the set of {\em contractible} homoclinic points. It is invariant under the action of $\phi$.

\vsp

Hohloch \cite{hohloch1} showed that there is a (relative) {\em Maslov index} $\mu(p, q) \in \Z$ for $p$, $q \in \mcH$ (for $[p] =[q]$). 
If we assume the intersections to be perpendicular and if we flip $+90^\circ$ at $q$ and $-90^\circ$ at $p$ we can identify $\mu(p,q)$ in our two dimensional setting with twice the winding number of the unit tangent vector of a loop starting in $p$, running through $[p,q]_u$ to $q$ and through $[p,q]_s$ back to $p$. We have $\mu(p, q)= \mu(\phi^n(p),\phi^n(q))$ for $n \in \Z$.
The (relative) Maslov index yields a {\em grading} $\mu: \mcH_{[x]} \to \Z $ via $\mu(p):=\mu(p,x)$ such that for contractible homoclinic points $p$ and $q$ holds
\beqs
\mu(p,q)= \mu(p,x) +\mu(x,q)=\mu(p,x)- \mu(q,x)= \mu(p)-\mu(q).
\eeqs
$\mcH$ and $\mcH_{[x]}$ are somehow {\em `too large'} sets in order to be used as generators for a Floer chain complex. But we will find now {\em finite} subsets which can serve as generator sets for a Floer theory. 
%
%
$p \in \mcH_{[x]} \setminus\{x\}$ is called {\em primary} if  
\beq
\label{primary}
]p,x[_s\ \cap \ ]p,x[_u\ \cap\ \mcH_{[x]} = \emptyset.
\eeq
The set of primary points is denoted by $\mcHpr$. These homoclinic points have the following important properties.

\begin{Lemma}[Hohloch \cite{hohloch1}, Remark 16, Lemma 17, and Remark 18]
\label{position prim}
\mbox{  }
\begin{enumerate}[(i)]
\item
$\mcHpr$ is invariant under $\phi$.
\item
Let $p$ be primary. Then all primary points lying in the intersection set of the same pair of intersecting branches as $p$ have a unique representative in $]p, \phi(p)]_s \ \cap\ ]p, \phi(p)]_u$.
\item
For a primary point $p$ holds $\mu(p)=\mu(p,x) \in \{\pm 1, \pm 2, \pm 3\}$, i.e. the Maslov index is bounded.
\item
If $W^s$ and $W^u$ intersect transversely then $\mcHprti:=\mcHpr \slash \Z$ is \textbf{finite}.
\end{enumerate}
\end{Lemma}

We denote the equivalence class of $p\in \mcHpr$ in $ \mcHprti= \mcHpr \slash \Z$ by $\langle p \rangle$. The homotopy class and the Maslov index $\mu$ pass to the quotient via $[\langle p \rangle ] := [p]$, $\mu(\langle p \rangle, \langle q \rangle):=\mu(p,q)$ and $\mu(\langle p \rangle):= \mu(p,x)$.

\vsp

Now we sketch the construction of Floer homology for the {\em finite} set $\mcHprti$.
Consider a fixed 2-gon $D$ in $\R^2$ with convex vertices at $(-1,0)$ and $(1,0)$.
Denote its lower edge by $B_u$ and its upper edge by $B_s$.
For $p$, $q \in \mcH$ with $\mu(p)-\mu(q)=1$, we define $\mcM(p,q)$ to be the space of smooth, immersed 2-gons $v : D \to M$ which are orientation preserving and satisfy
$v(B_u) \subset W^u$, $v(B_s) \subset W^s$, 
$v(-1,0)=p$ and $v(1,0)=q$.
Denote by $G(D)$ the group of orientation preserving diffeomorphisms of $D$ which preserve the vertices and set $\mcMhat(p,q):=\mcM(p,q) \slash G(D)$.


Endow each branch of the (un)stable manifolds with its `iteration-jump direction' as orientation and denote it by $o(${\small branch}$)$.
When working with distinct primary points $p$ and $q$ which lie in the same branch, denote said branch by $W_{pq}$.
For distinct primary points $p$, $q$ with $\mu(p,q)=1$ and $v \in \mcM(p,q) \neq \emptyset$ associate to $v(B_i)=[p,q]_i$ the orientation induced by the parametrization direction from $p$ to $q$ called $o_{pq}$.
We set
\beqs
m(p,q):= \left\{
\begin{aligned}
 1 &&& \mbox{if } \mu(p,q)=1, \ \mcM(p,q) \neq \emptyset, \ o(W_{pq})=o_{pq}, \\
 -1 &&& \mbox{if } \mu(p,q)=1, \ \mcM(p,q) \neq \emptyset, \ o(W_{pq}) \neq
o_{pq}, \\
0 &&& \mbox{otherwise}
\end{aligned}
\right.
\eeqs
which was proven to be welldefined in Section 3.2 of Hohloch \cite{hohloch1}.
For $\langle p
\rangle $, $\langle q \rangle \in \mcHprti$ set
$m(\langle p \rangle, \langle q \rangle):=\sum_{n \in \Z} m(p,\phi^n(q))$. We define the {\em Floer chain groups} and the {\em Floer boundary operator} via
\begin{gather*}
 C_k:=C_k(\phi, x; \Z):= \bigoplus_{\stackrel{\langle p \rangle \in \mcHprti}{\mu(\langle p
 \rangle)=k}} \Z \langle p \rangle,
\qquad
\del \langle p \rangle := \sum_{\stackrel{\langle q \rangle \in
    \mcHprti}{\mu(\langle q \rangle)=\mu(\langle p \rangle) -1}} m(\langle p
\rangle , \langle q \rangle) \langle q \rangle
\end{gather*}
on a generator $\langle p \rangle$ and extend $\del$ by linearity.
All groups have finite rank and, due to \refpositionprim, $C_k=0$ for $k \notin \{\pm 1, \pm 2, \pm 3\}$.

\begin{Theorem}[Hohloch \cite{hohloch1}, Theorem 23]
\label{phfh}
\mbox{ \ }
\begin{enumerate}[(i)]
 \item 
$\del \circ \del=0$, i.e. $(C_*, \del_*)$ is a chain complex and 
\beqs
H_k:= H_k( \phi, x; \Z):= \frac{\ker \del_k}{\Img \del_{k+1}}
\eeqs
is called primary homoclinic Floer homology of $\phi$ in $x$.
\item
We have $H_k=0$ for $k \neq \pm 1, \pm 2, \pm 3$.
\end{enumerate}
\end{Theorem}

The proofs of the welldefinedness of $\del$ and of $\del \circ \del=0$ involve the so-called breaking and gluing procedure which mainly relies on the classification of $\mcMhat(p,q)$ and of immersions of relative Maslov index 2. Certain parts of the proofs are of combinatorial nature whereas other parts make use of the iteration behaviour of $W^s \cap W^u$ and use classical dynamical results like Palis' $\lam$-Lemma \cite{palis}.

\vsp

Since $\mcHprti$ is finite and since also the sum in the definition of $\del$ is in fact finite, we conclude:

\begin{Remark}
 \label{compactsegment}
Primary Floer homology is completely determined by a finite number of primary homoclinic points located in (possibly large) compact segments of the (un)stable manifolds centered around the fixed point.
\end{Remark}

One aim of this paper is to determine the size of these compact segments.


\section{Computational complexity}

\label{sectionmorse}


\subsection{Generators and boundary operator}

For the computational complexity of primary homoclinic Floer homology, we have two main steps to analyse:
\begin{enumerate}[(i)]
 \item 
 Finding of the generators, i.e.\ the primary points.
 \item
 Computation of the boundary operator, i.e.\ finding the connecting immersions.
\end{enumerate}
There are two different aspects to pursue:
\begin{enumerate}[(a)]
 \item 
 {\em Theoretical knowledge:} Existence and welldefinedness of generators and immersions (proved in the previous work \cite{hohloch1}); enhancement and exact upper bounds of the `finding algorithm' for these generators and immersions (will be done in this section).
 \item
 {\em Numerical realization:} Actual computation of some examples by numerical methods. We aim at employing Wim Vanroose's numerical methods based on Newton-Krylov solvers (cf.\ for instance Schl\"omer $\&$ Avitabile $\&$ Vanroose \cite{schloemer-avitabile-vanroose}). This is an ongoing project with Wim Vanroose and will be the content of a future work.
\end{enumerate}

For symplectomorphisms on $\R^2$ with compact support, Pixton's work \cite[Theorem C]{pixton} assures the generic existence of homoclinic points for hyperbolic fixed points. On surfaces with genus, Oliveira \cite{oliveira1}, \cite{oliveira2} proved generic existence under certain natural conditions. Thus we are not talking about the empty set in the following statement.

\begin{Lemma}
\label{firstIntersection}
 The first intersection point found by 
 \begin{quote}
 \begin{enumerate}[i)]
  \item 
  starting at the hyperbolic fixed point and
  \item
  tracing simultaneously a branch of $W^u$ and a branch of $W^s$
 \end{enumerate}
 \end{quote}
is a primary homoclinic point.
\end{Lemma}

\begin{proof}
 Denote by $p$ the first intersection point found by starting at the hyperbolic fixed point $x$ and tracing simultaneously a branch of $W^u$ and a branch of $W^s$. Then by definition $]p,x[_s\ \cap\ ]p, x[_u = \emptyset$ which implies $p$ primary.
\end{proof}

In particular symplectomorphisms that are time-1 maps coming from a perturbation of an automomous system with a homoclinic orbit do have homoclinic points (see for instance the Melnikov method, cf.\ Guckenheimer $\&$ Holmes \cite[Chapter 4]{guckenheimer-holmes}).

Before we have a look at the positioning of generators within the branches note that symplectic diffeomorphisms are either $W$-orientation preserving or swap the stable branches as well as the unstable branches. If $\phi$ is a symplectomorphism then $\phi^2:= \phi \circ \phi$ is always $W$-orientation preserving.

\begin{Remark}[Hohloch \cite{hohloch1}, Remark 16 and Remark 18]
 Let $\phi$ be a symplectomorphism with hyperbolic fixed point $x$ and $p \in \mcH\setminus \{x\}$. Let $\phi$ be $W$-orientation preserving and let $W^s_p$ be the branch of $W^s$ containing $p$; analogously define $W^u_p$. Then $[p, \phi(p)[_s \ \cap\ [p, \phi(p)[_u \ \cap\ \mcHpr$ is a representative system of all equivalence classes $\langle q\rangle \in \mcHprti$ with $q \in W^s_p \cap W^u_p$. 
\end{Remark}

Now let us analyse the boundary operator. Its welldefinedness is based on the following result.

\begin{Proposition}[Hohloch \cite{hohloch1}, Proposition 26 and Lemma 27]
 \label{refineclassif}
Let $\phi \in \Symp(\R^2)$ with $x \in \Fix(\phi)$ hyperbolic. 
 \begin{enumerate}[1)]
  \item
  Let $p$, $q \in \mcHpr$ with $\mu(p,q)=1$ and $\mcM(p,q) \neq \emptyset$.
  Then the immersions in $\mcM(p,q)$ are in fact embeddings. In particular $]p,q[_s\ \cap\ ]p,q[_u\ = \emptyset$.
  \item
  Let $p$, $q \in \mcH_{[x]}$. Then there is $L \in \N_0$ such that for $n \in \Z$ with $\abs{n}>L$ we have $]p, \phi^n(q)[_s\ \cap\ ]p, \phi^n(q)[_u\ \neq \emptyset$. In particular, if the space $\mcM(p,\phi^n(q))$ is welldefined, then it is empty.
 \end{enumerate}
\end{Proposition}

In order to get an upper bound on the search depth of our future algorithm we have to determine the integer $L$ in the second item of \refrefineclassif. Therefore we need some additional notation:
Assume $\phi$ to be $W$-orientation preserving. On each branch of an (un)stable manifold, we introduce an {\em ordering via the `jump direction'} as follows: 
\begin{itemize}
\item 
Let $p$, $q \in \mcH$ lie in the same branch of $W^s$. We write $p<_s q$ (or $q>_s p$) if $[x, p]_s \supset [x, q]_s$. 
\item
Let $p$, $q \in \mcH$ lie in the same branch of $W^u$. We write $p<_u q$ (or $q>_u p$) if $[x, p]_u \subset [x, q]_u$. 
\end{itemize}
The segment $]p,  \infty[_u$ stands for all points $q >_u p$ and $]p,  - \infty[_s$ stands for all points $q <_s p$.

\begin{Theorem}
 \label{oneiterate}
 Let $\phi \in \Symp(\R^2)$ be $W$-orientation preserving with $x \in \Fix(\phi)$ hyperbolic. Let $p$, $q \in \mcHpr$, $\mu(p,q)=1$ and $\mcM(p,q) \neq \emptyset$.
 Then 
 \begin{enumerate}[1)]
  \item 
$\mcM(p, \phi^n(q)) = \emptyset$ for $n \notin\{ -1, 0, 1\}$.
  \item
 If $\mcM(p, \phi(q))\neq \emptyset$ then  $m(p, q) =-m(p, \phi(q))$ and $\mcM(p, \phi^{-1}(q)) = \emptyset$.
  \item
  If $\mcM(p, \phi^{-1}(q)) \neq \emptyset$ then $m(p, q) =-m(p, \phi^{-1}(q))$ and $\mcM(p, \phi(q)) = \emptyset$.
 \end{enumerate} 
\end{Theorem}

This determines the integer $L$ in \refrefineclassif\ as $L=1$. The proof of \refoneiterate\ is technical and lengthy:

\begin{proof}[Proof of \refoneiterate]
We have to check the list of the following possibilities; there is no shorter way up to our knowledge.
\begin{itemize}
 \item
 Four {\em cases} (denoted in the following by capital Roman numbers) check where the fixed point $x$ lies w.r.t.\ $]p, q[_s$ and $]p, q[_u$.
 \item
 Two {\em subcases} (denoted in the following by Arabic numbers) distinguish wether $p<_s q$ or not and $p<_u q$ or not.
 \item
 Four {\em subsubcases} (denoted in the following by small Latin characters) check if, for some $n \in \Z\setminus\{0\}$, the iterate $\phi^n(q)$ can lie in $]p,q[_s$ and/or $]p,q[_u$ or not. 
\end{itemize}
Certain cases are somewhat symmetric, but since we also want to determine signs $m(p, q)$ we opted for writing up the proof in full detail to avoid any confusion or even potential mistakes. Now let us begin with the proof.

\vspace{2mm}

\textbf{\textit{(I) Case $ x \in\ ]p, q[_s$ and $x \in\ ]p, q[_u$:}} This implies $ x \in\ ]p, q[_s\ \cap\ ]p, q[_u$ and, by \refrefineclassif, $\mcM(p,q) =\emptyset$ $\lightning$.

\vspace{2mm}

\textbf{\textit{(II) Case $x \notin\ ]p, q[_s$ and $x \notin\ ]p, q[_u$:}}

{\em (II.1) Subcase $p<_u q$ and $p<_s q$:}

{\em (II.1.a) Subsubcase $\phi^n(q) \in\ ]p,q[_s$ and $\phi^n(q) \in\ ]p,q[_u$:} 
This implies $\phi^n(q) \in\ ]p,q[_s\ \cap\ ]p,q[_u$ implying $\mcM(p,q) = \emptyset$ according to \refrefineclassif\ $\lightning$.

{\em (II.1.b) Subsubcase $\phi^n(q) \in\ ]p,q[_s$ and $\phi^n(q) \notin\ ]p,q[_u$:} 
This implies $n < 0$ meaning $\phi^n(q) \in\ ]p,x[_u$. Moreover $\phi^n(q) \in\  ]p,q[_s\ \subset\ ]p,x[_s$ implying $p$ not primary $\lightning$.

{\em (II.1.c) Subsubcase $\phi^n(q) \notin\ ]p,q[_s$ and $\phi^n(q) \in\ ]p,q[_u$:} 
This implies $n < 0$ implying $\phi^n(q) \in\ ]p, -\infty[_s$ and $p \in\ ]\phi^n(q), x[_s\ \cap\ ]\phi^n(q), x[_u$ such that $\phi^n(q)$ is not primary $\lightning$.

{\em (II.1.d) Subsubcase $\phi^n(q) \notin\ ]p,q[_s$ and $\phi^n(q) \notin\ ]p,q[_u$:} 
If $n >0$ then $\phi^n(q) \in\ ]q, \infty[_u$ and $\phi^n(q) \in\ ]q, x[_s$ and in particular $q \in\ ] p, \phi^n(q)[_s\ \cap\ ] p, \phi^n(q)[_u$ implying $\mcM(p,\phi^n(q)) = \emptyset$ by \refrefineclassif.

If $n=-1$, the space $\mcM(p, \phi^{-1}(q))$ may be nonempty, cf.\ the positioning of $\phi^{-1}(q)=\bigstar$ in Figure \ref{good1} and the absence of additional `interfering' intersection points (otherwise $\bigstar$ may not be primary). Moreover, if $\mcM(p, \phi^{-1}(q)) \neq \emptyset$, we observe $m(p, q) =-m(p, \phi^{-1}(q))$.

If $n< -1$, then we get from the precendent case $n=-1$ that $\phi^{-1}(q) \in\ ]p, - \infty[_s$ and also $\phi^{-1}(q) \in\ ]x, p[_u$. Thus holds, for $n>-1$, in particular $\phi^{-1}(q) \in\ ]\phi^n(q), p[_s\ \cap\ ]\phi^n(q), p[_u$ such that $\mcM(p, \phi^n(q))= \emptyset$.

\begin{figure}[h]
\begin{center}

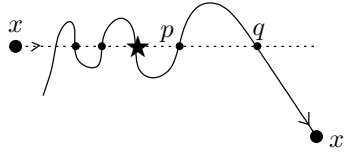

\caption{In this and the following figures, we sketch the cases with either $\mcMhat(p,\phi(q)) \neq \emptyset$ or $\mcMhat(p,\phi^{-1}(q)) \neq \emptyset$. The symbol $\bigstar$ stands for $\phi(q)$ resp.\ $\phi^{-1}(q)$. The unstable manifold is drawn as dotted line and the stable manifold als black line and, for the sake of easier drawing, $x$ is split into two copies (as done in most of the pictures in Hohloch \cite{hohloch1, hohloch2}).}
\label{good1}

\end{center}
\end{figure}

{\em (II.2) Subcase $p<_u q$ and $p>_s q$:}
This implies immediately $p \in\ ] x, q[_s\ \cap\ ]x, q[_u$ such that $q$ is not primary $\lightning$.

{\em (II.3) Subcase $p>_u q$ and $p<_s q$:}
We find $q \in\ ] x, p[_s\ \cap\ ]x, p[_u$ such that $p$ is not primary $\lightning$.

{\em (II.4) Subcase $p>_u q$ and $p>_s q$: This is quite analogous to (II.1), but nevertheless:} 

{\em (II.4.a) Subsubcase $\phi^n(q) \in\ ]p,q[_s$ and $\phi^n(q) \in\ ]p,q[_u$:} 
This implies $\phi^n(q) \in\ ]p,q[_s\ \cap\ ]p,q[_u$ such that $\mcM(p,q)= \emptyset $ by \refrefineclassif\ $\lightning$.

{\em (II.4.b) Subsubcase $\phi^n(q) \in\ ]p,q[_s$ and $\phi^n(q) \notin\ ]p,q[_u$:} 
Thus $n>0$ and $\phi^n(q) \in\ ]p, \infty[_u$ such that $p \in\ ]\phi^n(q), x[_s \ \cap\ ]\phi^n(q), x[_u$ hindering $\phi^n(q)$ from being primary $\lightning$.

{\em (II.4.c) Subsubcase $\phi^n(q) \notin\ ]p,q[_s$ and $\phi^n(q) \in\ ]p,q[_u$:} 
Thus $n>0$ and $\phi^n(q) \in \ ]p,x[_s$ such that $\phi^n(q) \in \ ]p,x[_s \ \cap\ ]p,x[_u$ implying $p$ not primary $\lightning$.

{\em (II.4.d) Subsubcase $\phi^n(q) \notin\ ]p,q[_s$ and $\phi^n(q) \notin\ ]p,q[_u$:} 
For $n=1$, we may have $\mcM(p,\phi(q)) \neq \emptyset$ 
cf.\ the positioning of $\phi(q)=\bigstar$ in Figure \ref{good2} and the absence of additional `interfering' intersection points (otherwise $\bigstar$ may not be primary).
In that case, $m(p, q) =-m(p, \phi^{-1}(q))$. For $n>1$, we find $\phi(q) \in\ ]p, \phi^n(q)[_s\ \cap\ ] p, \phi^n(q)[_u$ such that $\mcM(p, \phi^n(q))= \emptyset$.
If $n<0$ then $q \in \ ]\phi^n(q), p[_s\ \cap\ ]\phi^n(q), p[_u$ and $\mcM(p, \phi^n(q))= \emptyset$.

\begin{figure}[h]
\begin{center}

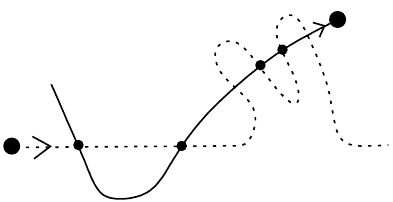

\caption{}
\label{good2}

\end{center}
\end{figure}

\textbf{\textit{(III) Case $x \notin\ ]p, q[_s$ and $x \in\ ]p, q[_u$:}}
Since $\phi$ is $W$-orientation preserving, $\phi^n(q)$ lies always in the same branch as $q$. Keep this in mind in the following.

{\em (III.1) Subcase $p<_s q$:}

{\em (III.1.a) Subsubcase $\phi^n(q) \in\ ]p,q[_s$ and $\phi^n(q) \in\ ]p,q[_u$:} 
Then $\phi^n(q) \in\ ]p,q[_s\ \cap\  ]p,q[_u$ and $\mcM(p,q) = \emptyset$ $\lightning$.

{\em (III.1.b) Subsubcase $\phi^n(q) \in\ ]p,q[_s$ and $\phi^n(q) \notin\ ]p,q[_u$:} 
Since by assumption $\phi^n(q) \notin\ ]p,q[_u\ \supset \ ]x, q[_u$ we conclude $n>0$ and note $q \in \ ]\phi^n(q), x[_s\ \cap\ ]\phi^n(q), x[_u$ which prevents $\phi^n(q)$ from being primary $\lightning$.

{\em (III.1.c) Subsubcase $\phi^n(q) \notin\ ]p,q[_s$ and $\phi^n(q) \in\ ]p,q[_u$:} 
We note $\phi^n(q) \in\ ]x, q[_u\ \subset\ ]p,q[_u$ concluding $n<0$ and $\phi^n(q) \in\ ]p, -\infty[_s$. For $n=-1$ we may have $\mcM(p, \phi(q)) \neq \emptyset$ 
cf.\ the positioning of $\phi^{-1}(q)=\bigstar$ in Figure \ref{good3} and the absence of additional `interfering' intersection points (otherwise $\bigstar$ may not be primary).
Moreover, we find in this case $m(p, q) =-m(p, \phi^{-1}(q))$. If $n< -1$ then $\phi^n(q) \in\ ]\phi^{-1}(q), x[_s\ \cap\  ]\phi^{-1}(q), x[_u$ preventing $\phi^n(q)$ from being primary $\lightning$.

{\em (III.1.d) Subsubcase $\phi^n(q) \notin\ ]p,q[_s$ and $\phi^n(q) \notin\ ]p,q[_u$:} 
Since $\phi^n(q) \notin\ ]p,q[_u$ we conclude $\phi^n(q) \in\ ]q, \infty[_u$. Thus $n>0$ and we note $q \in\ ]p, \phi^n(q)[_s\ \cap\ ]p, \phi^n(q)[_u$ implying $\mcM(p, \phi^n(q))= \emptyset$.

\begin{figure}[h]
\begin{center}

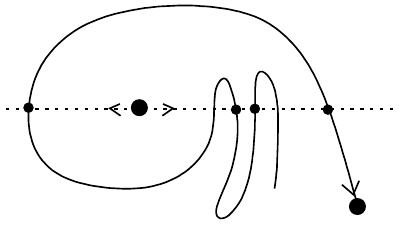

\caption{}
\label{good3}

\end{center}
\end{figure}

{\em (III.2) Subcase $p>_s q$:} 

{\em (III.2.a) Subsubcase $\phi^n(q) \in\ ]p,q[_s$ and $\phi^n(q) \in\ ]p,q[_u$:} 
Then $\phi^n(q) \in\ ]p,q[_s\ \cap\ ]p,q[_u$, thus $\mcM(p,q) = \emptyset $ $ \lightning$.

{\em (III.2.b) Subsubcase $\phi^n(q) \in\ ]p,q[_s$ and $\phi^n(q) \notin\ ]p,q[_u$:} 
We conclude $n>0$ and $\phi^n (q) \in\ ]q, \infty[_u$. Thus there is $k<0$ with $\phi^k(p) \in\ ]\phi^n(q), q[_s\ \subset\ ]p,q[_s$ implying $\phi^k(p) \in\ ]p, q[_s\ \cap\ ]p,q[_u$ such that $\mcM(p,q) = \emptyset$ $\lightning$.

{\em (III.2.c) Subsubcase $\phi^n(q) \notin\ ]p,q[_s$ and $\phi^n(q) \in\ ]p,q[_u$:} 
We deduce $n<0$ and, since $\phi^n(q) \notin\ ]p,q[_s$, we have $\phi^n(q) \in\ ]q, -\infty[_s$. Thus there is $k<0$ with $\phi^k(p) \in\ ]q, \phi^n(q)[_s\ \subset\ ]p, \phi^n(q)[_s$. Moreover $\phi^k(p) \in\ ]p, \phi^n(q)[_u$ thus $\mcM(p, \phi^n(q)) = \emptyset$.

{\em (III.2.d) Subsubcase $\phi^n(q) \notin\ ]p,q[_s$ and $\phi^n(q) \notin\ ]p,q[_u$:} 
If $n>0$ then $\phi^n(q) \in\ ]q, \infty[_u$ and $\phi^n(q) \in\ ]p, x[_s$. For $n=1$ we may have $\mcM(p, \phi(q)) \neq \emptyset$,  
cf.\ the positioning of $\phi(q)=\bigstar$ in Figure \ref{good4} and the absence of additional `interfering' intersection points (otherwise $\bigstar$ may not be primary).
In that case, we observe $m(p, q) =-m(p, \phi(q))$. 
If $n>1$ then $\phi(q) \in\ ]q, \phi^n(q)[_u\ \subset\ ]p,\phi^n(q)[_u$. Moreover $\phi(q) \in\ ]p, \phi^n(q)[_s$ such that $\mcM(p, \phi^n(q))= \emptyset$.
If $n<0$ then $\phi^n(q) \in\ ]q, x[_u$ and $\phi^n(q) \in\ ]q, - \infty[_s$. There exists $k<0$ with $\phi^k(p) \in\ ]q, \phi^n(q)[_s\ \subset\ ]p, \phi^n(q)[_s$. Moreover $\phi^k(p) \in\ ]p, x[_u\ \subset\ ]p, \phi^n(q)[_u$ thus $\mcM(p, \phi^n(q))= \emptyset$.

\begin{figure}[h]
\begin{center}

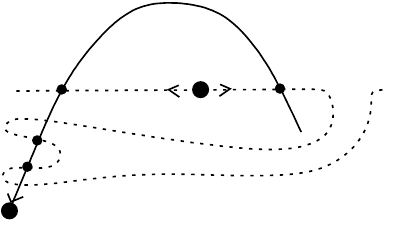

\caption{}
\label{good4}

\end{center}
\end{figure}

\textit{\textbf{(IV) Case $x \in\ ]p, q[_s$ and $x \notin\ ]p, q[_u$:} This is quite analogous to (III), but nevertheless:}

Since $\phi$ is $W$-orientation preserving, $\phi^n(q)$ lies always in the same branch as $q$. Keep this in mind in the following.

{\em (IV.1) Subcase $p<_u q$:}

{\em (IV.1.a) Subsubcase $\phi^n(q) \in\ ]p,q[_s$ and $\phi^n(q) \in\ ]p,q[_u$:} 
Then $\phi^n(q) \in\ ]p,q[_s\ \cap\ ]p,q[_u$ and $\mcM(p,q) = \emptyset$ $\lightning$.

{\em (IV.1.b) Subsubcase $\phi^n(q) \in\ ]p,q[_s$ and $\phi^n(q) \notin\ ]p,q[_u$:} 
We conclude $n>0$ and $\phi^n(q) \in\ ]x,q[_s$. There is $k>0$ such that $\phi^k(p) \in\ ]q, \phi^n(q)[_u\ \subset\ ]p, \phi^n(q)[_u$. Moreover $\phi^k(p) \in\ ]p, x[_s\ \subset\ ]p, \phi^n(q)[_s$ implying $\mcM(p, \phi^n(q)) = \emptyset$.

{\em (IV.1.c) Subsubcase $\phi^n(q) \notin\ ]p,q[_s$ and $\phi^n(q) \in\ ]p,q[_u$:} 
We conclude $n<0$ and $\phi^n(q)\in\ ]q, -\infty[_s$. There is $k>0$ with $\phi^k(p) \in\ ]\phi^n(q), q[_u\ \subset\ ]p,q[_u$. Moreover $\phi^k(p) \in\ ]p,x[_s\ \subset\ ]p,q[_s$ implying $\mcM(p,q) = \emptyset $ $\lightning$.

{\em (IV.1.d) Subsubcase $\phi^n(q) \notin\ ]p,q[_s$ and $\phi^n(q) \notin\ ]p,q[_u$:} 
We conclude $n<0$ and note $\phi^n(q) \in\ ]q, -\infty[_s$ and $\phi^n(q) \in\ ] x, p[_u$. For $n=-1$, we may have $\mcM(p, \phi^{-1}(q)) \neq \emptyset$
cf.\ the positioning of $\phi^{-1}(q)=\bigstar$ in Figure \ref{good5} and the absence of additional `interfering' intersection points (otherwise $\bigstar$ may not be primary).
We observe $m(p, q) =-m(p, \phi^{-1}(q))$. If $n<-1$ then $\phi^{-1}(q) \in\ ]\phi^n(q), p[_s\ \cap\ ]\phi^n(q), p[_u$ implying $\mcM(p, \phi^n(q)) =\emptyset$.

\begin{figure}[h]
\begin{center}

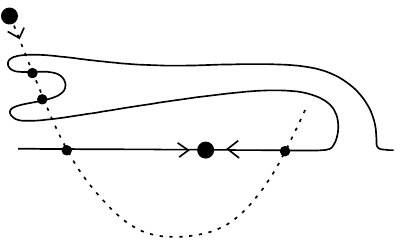

\caption{}
\label{good5}

\end{center}
\end{figure}

{\em (IV.2) Subcase $p>_u q$:}

{\em (IV.2.a) Subsubcase $\phi^n(q) \in\ ]p,q[_s$ and $\phi^n(q) \in\ ]p,q[_u$:} 
Then $\phi^n(q) \in\ ]p,q[_s\ \cap\ ]p,q[_u$ and $\mcM(p,q) = \emptyset$ $\lightning$.

{\em (IV.2.b) Subsubcase $\phi^n(q) \in\ ]p,q[_s$ and $\phi^n(q) \notin\ ]p,q[_u$:} 
We conclude $n>0$ and $\phi^n(q) \in\ ]q, \infty[_u$. 
If $n=1$ we may have $\mcM(p, \phi(q)) \neq \emptyset$ 
cf.\ the positioning of $\phi(q)=\bigstar$ in Figure \ref{good6} and the absence of additional `interfering' intersection points (otherwise $\bigstar$ may not be primary).
In that case, $m(p, q) =-m(p, \phi(q))$ as sketched in Figure \ref{good6}. For $n>1$, there is $k>0$ such that $\phi^k(p) \in\ ]\phi(q), \phi^n(q)[_u\ \subset\ ]p, \phi^n(q)[_u$. Moreover $\phi^k(p) \in\ ]p,x[_s\ \subset\ ]p, \phi^n(q)[_s$ implying $\mcM(p, \phi^n(q)) = \emptyset$.

{\em (IV.2.c) Subsubcase $\phi^n(q) \notin\ ]p,q[_s$ and $\phi^n(q) \in\ ]p,q[_u$:} 
We deduce $n>0$, but then $\phi^n(q) \in\ ]x, q[_s\ \subset\ ]p,q[_s$ $\lightning$.

{\em (IV.2.d) Subsubcase $\phi^n(q) \notin\ ]p,q[_s$ and $\phi^n(q) \notin\ ]p,q[_u$:}
If $n<0$ then $q \in\ ]p, \phi^n(q)[_u$ and also  $q \in\ ]p, \phi^n(q)[_s$ implying $\mcM(p, \phi^n(q)) = \emptyset$. If $n>0$ then $\phi^n(q) \in\ ]x, q[_s\ \subset\ ]p,q[_s$ $\lightning$.
\end{proof}

\begin{figure}[h]
\begin{center}

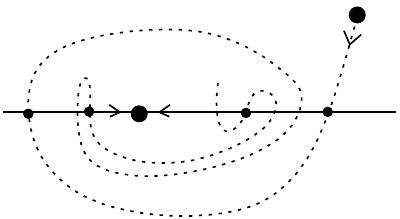

\caption{}
\label{good6}

\end{center}
\end{figure}

The proof of \refoneiterate\ implies in particular:

\begin{Corollary}
 \label{place}
 Let $\phi \in \Symp(\R^2)$ be $W$-orientation preserving with $x \in \Fix(\phi)$ hyperbolic. Let $p$, $q \in \mcHpr$, $\mu(p,q)=1$ and $\mcM(p,q) \neq \emptyset$. Then
 \begin{enumerate}[1)]
  \item 
  If $x \notin\ ]p,q[_s$ and $x \notin\ ]p,q[_u$ then $q \in\ ]\phi^{-1}(p), \phi(p)[_s \ \cap\ ]\phi^{-1}(p), \phi(p)[_u$.
  \item
  If $x \notin\ ]p,q[_s$ and $x \in\ ]p,q[_u$ then $q \in \ ]\phi^{-1}(p), \phi(p)[_s$.
  \item
  If $x \in\ ]p,q[_s$ and $x \notin\ ]p,q[_u$ then $q \in \ ]\phi^{-1}(p), \phi(p)[_u$.
  \item
  If $x \in\ ]p,q[_s$ and $x \in\ ]p,q[_u$ then $\mcM(p,q) = \emptyset$, i.e.\ this case does not occur.
 \end{enumerate}
\end{Corollary}


\subsection{The improved algorithm}

The calculation of primary homoclinic Floer homology can be optimized as summarized in the following algorithm. The results from the previous section enter in {\em Step 5)} when the necessary data for the boundary operator are determined.

\begin{enumerate}[1)]
 \item 
 Check all four pairs of branches of the stable and unstable manifolds for the existence of one intersection point using the method in \reffirstIntersection. If one intersection point $p$ is found it is primary according to \reffirstIntersection.
 \item
  For all primary points $p_i$, where $1 \leq i \leq 4$, found in {\em Step 1)}, determine all intersection points in $]p_i, \phi(p_i)[_s\ \cap\ ]p_i, \phi(p_i)[_u$. For transversely intersecting stable and unstable manifolds, this is a finite number $L_i$.
  \item
  For all $1 \leq i \leq 4$ and all $L_i $ intersection points of $[p_i, \phi(p_i)[_s\ \cap\ [p_i, \phi(p_i)[_u$ found in {\em Step 2)}, determine the $K_i \leq L_i$ primary ones and remember them. According to \refpositionprim, we find that way exactly one representative of each equivalence class of primary points.
   \item
  For all $1 \leq i \leq 4$ and all $K_i $ primary intersection points $p$ in $[p_i, \phi(p_i)[_s\ \cap\ [p_i, \phi(p_i)[_u$ found according to {\em Step 3)}, determine their Maslov index $\mu(p):=\mu(p,x)$. This enables us to define the chain groups of the complex.
  \item
  For all $1 \leq i \leq 4$ and all $K_i$ primary points $p$ in $[p_i, \phi(p_i)[_s\ \cap\ [p_i, \phi(p_i)[_u$ and all primary points $q$ in $]\phi^{-1}(p), \phi(p)[_s$ and $ ]\phi^{-1}(p), \phi(p)[_u$ with $\mu(q)= \mu(p)-1$, determine 
  \begin{enumerate}[(i)]
  \item
  if $]p,q[_s\ \cap\ ]p,q[_u = \emptyset$,
  \item
  and, if yes, calculate the sign $m(p,q)$.
  \end{enumerate}
  This relies on \refoneiterate\ and provides all necessary information in order to calculate the boundary operator in the next step.
  \item
  Using the already gathered information, we can now calculate
  $$\del \langle p \rangle = \sum_{\stackrel{\langle q \rangle \in
    \mcHprti}{\mu(\langle q \rangle)=\mu(\langle p \rangle) -1}} m(\langle p
  \rangle , \langle q \rangle) \langle q \rangle.
  $$
  \item
  Calculate $\ker \del$ and $\Img \del$ and $H_k:=H_k( \phi, x; \Z):= \frac{\ker \del_k}{\Img \del_{k+1}}$.
\end{enumerate}

The implementation and evaluation of this algorithm with numerical methods is an ongoing project with Wim Vanroose.


\section{Torsion freeness and Morse inequalities}

\label{torsionMorse}

Apart from speeding up the calculation of primary homoclinic Floer homology, \refoneiterate\ has also purely algebraic applications: as we will see in this section, it implies that primary homoclinic Floer homology is torsion-free and, eventually, there are Morse type inequalities for homoclinic points.


\subsection{Classical Morse inequalities}

\label{classicalMorse}

In this paragraph, we briefly sketch the approach to Morse homology via the Morse-Smale-Witten complex (cf.\ for instance Schwarz \cite{schwarz}).

\vsp

Let $N$ be a closed $m$-dimensional manifold. $f: N \to \R$ is a {\em Morse function} if, for all critical points $\Crit(f)=\{p \in N \mid Df|_p=0\}$, the Hessian $D^2f|_p$ of $f$ is nondegenerate. The {\em Morse index} $\Ind(p)$ of $p \in \Crit(f)$ is the number of negative eigenvalues of $D^2f|_p$.
The {\em $k$th integral Morse chain group} is defined by
\beqs
C^M_k(N, f; \Z) := \bigoplus_{\stackrel{p \in \Crit(f)}{\Ind(p)=k}} \Z p ,
\eeqs
i.e.\ it is the free abelian group generated by all critical points of index $k$. We abbreviate
$$c_k:= \rk C^M_k(N, f; \Z) =\abs{\{ p \in \Crit(f) \mid \Ind(p)=k \}}.$$ 
For a Morse function $f$ and a Riemannian metric $g$ on $N$, the {\em negative gradient flow} is the flow of the equation
$$ \gadot = - \grad_g f(\ga).$$
Such a pair $(f,g)$ is called {\em Morse-Smale} if the intersection of stable and unstable manifolds of critical points is always transverse. Under these conditions, the space of trajectories `joining' $p^- \in \Crit(f)$ to $p^+\in \Crit(f)$ given by
\begin{align*} 
\mcM(f,g, p^-, p^+) := \left\{ \ga: \R \to N \left| 
\begin{aligned}
& \gadot(t) = - \grad_g f(\ga(t)), \\ 
& \lim_{t \to - \infty}\ga(t)=p^-,\ \lim_{t \to  \infty}\ga(t)=p^+
\end{aligned}
\right. \right\} 
\end{align*}
has dimension $\Ind(p^-)-\Ind(p^+)$. It carries the action 
$$\R \x \mcM(f,g, p^-, p^+) \to \mcM(f,g, p^-, p^+), \quad (s,\ga) \mapsto \ga( \cdot + s).$$
Dividing by this action yields $\mcMhat(f,g, p^-, p^+):= \mcM(f,g, p^-, p^+)\slash \R$ which has dimension $\Ind(p^-)-\Ind(p^+)-1$. For critical points $p^-$ and $p^+$ with $\Ind(p^-)- \Ind(p^+)=1$, the space $\mcMhat(f,g, p^-, p^+)$ has dimension zero and has cardinality 
$\abs{\mcMhat(f,g, p^-, p^+)} < \infty$.
The trajectory spaces can actually be coherently endowed with an orientation which induces a {\em signed cardinality} 
$$\abs{\mcMhat(f,g, p^-, p^+)}_\Z \in \{\pm  \abs{\mcMhat(f,g, p^-, p^+)} \}.$$
This allows to define the boundary operator 
\begin{align*}
& \del^M_*: C^M_*(N, f; \Z) \to C^M_{*-1}(N, f; \Z), \\
& \del^M_* p^-:= \sum_{ \stackrel{p^+ \in \Crit(f)}{\Ind(p^+)=*-1}} \abs{\mcMhat(f,g, p^-, p^+)}_\Z\ p^+
\end{align*}
on the generators; it extends by linearity. It holds $\del^M_* \circ \del^M_{*+1} = 0$ such that $(C^M_*(N, f; \Z), \del^M_*)$ is a chain complex whose induced homology 
$$ H^M_*(N; \Z)= H^M_*(N, f,g ; \Z) := \ker( \del_*^M) \slash \Img( \del^M_{*+1})$$
is called {\em Morse homology}. It is in fact independent of $f$ and $g$ and isomorphic to the singular homology of $N$.

\vsp

Every finitely generated abelian group $G$ has a direct sum decomposition $G=F\oplus T$ where $F$ is a finitely generated free subgroup and $T$ is a unique torsion subgroup (that are all elements of finite order in the group $G$). The {\em rank of $G$} is denoted by $\rk G$ and is defined as the rank of $F$. The {\em torsion rank of $G$} is the minimal number of cyclic subgroups of whose direct sum $T$ is a subgroup.

\vsp

$C^M_*(N,F,G;\Z)$ is by definition a finitely generated free abelian group and $\ker( \del_*^M)$ and $\Img( \del_{*})$ are as its subgroups also finitely generated and free abelian. The Morse homology groups $H^M_*(N; \Z) = \ker( \del_*^M) \slash \Img( \del^M_{*+1})$ are as quotient groups certainly abelian, but not necessarily freely generated, i.e.\ they may have a torsion subgroup. 

\vsp

Standard examples for torsion in homology groups are the higher dimensional real projective spaces. If we work with $\Q$- or $\R$-coefficients, there is never torsion due to the universal coefficient theorem of homology.

\vsp

If we denote by $h_k:= \rk H^M_k(N;\Z)$ the rank and by $t_k$ the torsion rank of $H^M_k(N;\Z)$, then we certainly have $h_k \leq c_k$ for $0 \leq k \leq n$. A closer look leads to the so-called {\em Morse inequalities} (cf.\ also Postnikov $\&$ Rudyak \cite{postnikov-rudyak}):
\begin{align*}
 & h_k +t_k + t_{k-1}  \leq c_k  \quad \mbox{for } 0 \leq k \leq n \mbox{ with } t_{-1}:=0, \\
 & \sum_{i=0}^l (-1)^{l-i} h_i  \leq \sum_{i=0}^l (-1)^{l-i} c_i   \quad \mbox{for } 0 \leq l \leq n, \\
 & \sum_{i=0}^n (-1)^{n-i} h_i = \sum_{i=0}^n (-1)^{n-i} c_i, \\ 
 & \sum_{i=0}^n (-1)^{i} c_i = \chi(N)
\end{align*}
where $\chi(N)$ is the Euler characteristic of $N$.


\subsection{Torsion freeness and Morse inequalities for homoclinic Floer homology}

We now will show that primary homoclinic Floer homology is torsion-free and we will present Morse type inequalities for primary homoclinic points using the framework of primary homoclinic Floer homology.

\begin{Theorem}[{\bf Torsion freeness}]
\label{notorsion}
Let $\phi$ be a $W$-orientation preserving symplectomorphism on $\R^2$ or on a closed surface of genus greater than zero.
Then the primary homoclinic Floer groups $H_*(\phi, x; \Z)$ are free, i.e.\ their torsion subgroups are trivial.
\end{Theorem}

The universal coefficient theorem in homological algebra describes by means of torsion the dependence of homology groups on the chosen coefficient ring. Thus torsion-freeness implies that, for instance, homoclinic Floer homology computed with $\Z$-coefficients is the same as with $\Q$- or $\R$-coefficients.

Before we start with the proof of \refnotorsion\ we recall the following fact about quotients of free groups.

\begin{Lemma}[Baumslag $\&$ Chandler \cite{baumslag-chandler}, Corollary 6.17]
 \label{quotient}
 Let $G$ be a free abelian group with basis $g_1, \dots, g_l$ and let $H$ be the free abelian group generated by $u_1g_1, \dots, u_l g_l $ where $u_i \in \Z$. Then $G \slash H$ is a direct sum of cyclic groups of order $\uti_1, \dots, \uti_l$ where $\uti_i= \infty$ if $u_i=0$ and $\uti_i=\abs{u_i}$ otherwise.
\end{Lemma}

We are now able to prove \refnotorsion.

\begin{proof}[{\bf Proof of \refnotorsion}]
It is enough to consider symplectomorphisms on $\R^2$ since, in the case of a closed surface of genus greater than zero, we may work on the universal cover as in Hohloch \cite{hohloch1}.

{\em Step 1:}
Let $p$ and $q$ be primary homoclinic points of relative index one. As shown in Hohloch \cite{hohloch1} in the text between Definition 9 and Definition 10, the space $\mcMhat(p,q)$ is either empty or contains exactly one element. According to \refoneiterate, there are either exactly zero, exactly one or exactly two exponents such that $\mcMhat(p, \phi^n(q))$ is nonempty and, in the last case, the signs $m(p, \cdot)$ have opposite sign. Thus we find 
\begin{align*}
 m(\langle p \rangle, \langle q \rangle) \in \{-1, 0, +1\}
\end{align*}
for all primary $p$ and $q$ with $\mu(p,q)=1$. 
This means that the coefficients of $\langle q \rangle$ in the sum
\beqs
\del \langle p \rangle = \sum_{\stackrel{\langle q \rangle \ primary}{\mu(\langle q \rangle)=\mu(\langle p \rangle)-1}} m(\langle p \rangle ,\langle q \rangle) \langle q \rangle
\eeqs
have never values different from $\{+1, 0, -1\}$. In particular, there is no common divisor of all the $m(\langle p \rangle ,\langle q \rangle)$ different from $\pm 1$.

{\em Step 2:} 
$$C_*= \bigoplus_{\stackrel{\langle p \rangle \ primary}{\Ind(\langle p \rangle)=*}} \Z \langle p \rangle$$
is a finitely generated abelian group such that its subgroups 
$$\Img \del_{*+1} < \ker \del_* < C_*$$ 
are also finitely generated and free with 
$$\rk (\Img \del_{*+1}) \leq \rk (\ker \del_*) \leq \rk (C_*).$$ 
In {\em Step 1} we saw that the generators of $\Img \del_{*+1}$ are of the form 
\beqs
\del \langle p \rangle = \sum_{\stackrel{q \ primary}{\mu(q)=\mu(p)-1}} m(\langle p \rangle ,\langle q \rangle) \langle q \rangle
\eeqs
with $m(\langle p \rangle , \cdot) \in \{0, \pm 1 \}$. Among these generators we choose a basis of $\Img \del_{*-1}$. These vectors are never a multiple  with absolute value of the multiplier greater than one of a basis vector of $\ker (\del_*)$. According to \refquotient, the quotient has no nontrivial cyclic subgroups of finite order, i.e.\ no torsion.
\end{proof}

Recall that, by \refpositionprim, the Maslov index of a primary point $p$ satisfies $\mu(p) \in \{ \pm1, \pm 2, \pm 3\}$. Therefore only primary homoclinic Floer chain groups $C_k$ with $k \in  \{ \pm1, \pm 2, \pm 3\}$ can be nontrivial such that the complex looks like
\begin{align*}
\cdots \rightarrow 0 \stackrel{\del_4}{\longrightarrow} C_3  \stackrel{\del_3}{\longrightarrow} C_2 \stackrel{\del_2}{\longrightarrow} C_1 \stackrel{\del_1}{\longrightarrow} 0 \stackrel{\del_{0}}{\longrightarrow} C_{-1} \stackrel{\del_{-1}}{\longrightarrow} C_{-2}\stackrel{\del_{-2}}{\longrightarrow} C_{-3} \stackrel{\del_{-3}}{\longrightarrow} 0 \rightarrow \cdots
\end{align*}
We have
\beqs
 \rk \ker \del_k + \rk \Img \del_k = \rk C_k =: \mfc_k
\eeqs
and we set $\mfh_k:=\rk H_k$. Due to \refnotorsion, the torsion rank of $H_k$ vanishes and we find

\begin{Theorem}[{\bf Homoclinic Morse inequalities}]
\label{homoclinicMorse}
For the rank of the primary homoclinic Floer chain and homology groups holds:
\begin{enumerate}[1)]
 \item 
$ \mbox{For } k \in \Z: \quad  \mfh_k  \leq \mfc_k.$
  \item
$ \mbox{For } j <-3 \mbox{ and } 3 < l: \quad  \sum_{i=j}^l  \mfc_i = |\mcHprti|.$
  \item
$ \mbox{For } j, l \in \Z,\ j \leq l: \quad  \sum_{i= j}^l \mfh_i \leq \sum_{i=j}^l  \mfc_i \leq |\mcHprti|.$
  \item
$ \mbox{For } j, l \in \Z,\ j \leq l,\ j\leq -3 : \quad \sum_{i= j}^l (-1)^{l-i} \mfh_i \leq \sum_{i=j}^l (-1)^{l-i} \mfc_i .$
\end{enumerate}
\end{Theorem}

\begin{proof}
Due to \refnotorsion, the torsion rank vanishes. We estimate
\begin{enumerate}[1)]
 \item 
 
 $\mfh_k = \rk \ker \del_k - \rk \Img \del_{k+1} \leq   \rk \ker \del_k \leq \rk \ker \del_k + \rk \Img \del_k = \mfc_k$.

  \item
  
For $j <-3 $ and $3 < l$ holds: 
$$| \mcHprti | =\sum_{i=-3}^3 | \{ \langle p \rangle  \in \mcHprti \mid \mu(p)= i\} | = \sum_{i=-3}^3 \mfc_i = \sum_{i=j}^l  \mfc_i .$$

\item

follows from 1) and 2).

\item

Let $j, l \in \Z$ with $ j \leq l$. W.l.o.g. we assume $j=-3$. Keep in mind that thus $\rk \Img \del_j =0$. First consider the case $(l-j)$ even. Then we have
\begin{align*}	
& (-1)^{l-j} \mfc_j + (-1)^{l-j-1} \mfc_{j+1} + (-1)^{l-j-2} \mfc_{j+2} + \dots + (-1)^{l-l} \mfc_l \\
& = \mfc_j - \mfc_{j+1} + \mfc_{j+2} - \dots + \mfc_l \\
&  = \rk \ker \del_j + \rk \Img \del_j - \rk \ker \del_{j+1} - \rk \Img \del_{j+1} +  \rk \ker \del_{j+2} + \rk \Img \del_{j+2} \\
&  \quad - \dots + \rk \ker \del_l + \rk \Img \del_l \\
& = 0 + (\rk \ker \del_j -  \rk \Img \del_{j+1}) - (  \rk \ker \del_{j+1} -  \rk \Img \del_{j+2} )  \\
& \quad + (\rk \ker \del_{j+2} - \rk \Img \del_{j+3}) - \dots + ( \rk \ker \del_l -  \rk \Img \del_{l+1}) +  \rk \ker \del_{l+1} \\
& = (-1)^{l-j} \mfh_j + (-1)^{l-j-1} \mfh_{j+1} + (-1)^{l-j-2} \mfh_{j+2} + \dots + (-1)^{l-l} \mfh_l  +  \rk \ker \del_{l+1} \\
&  \geq (-1)^{l-j} \mfh_j + (-1)^{l-j-1} \mfh_{j+1} + (-1)^{l-j-2} \mfh_{j+2} + \dots + (-1)^{l-l} \mfh_l.
\end{align*}
The case $(j-l)$ odd follows similarly which finishes the proof.
\end{enumerate}
\end{proof}


\end{document}